\documentclass[12pt]{article}
\usepackage{e-jc}

\usepackage{amssymb}
\usepackage{amsmath,amssymb,amsfonts,amsthm,latexsym}
\usepackage{mathrsfs}
\usepackage{color}
\usepackage{eucal}%    caligraphic-euler fonts: \mathcal{ }
\usepackage[all]{xypic}
\usepackage{xspace}
\usepackage{graphicx}
\usepackage{bbm}

\newtheorem{theorem}{Theorem}

\newtheorem{lemma}{Lemma}
\newtheorem{proposition}{Proposition}

\newtheorem{definition}{Definition}

\numberwithin{equation}{section}

\date{\dateline{May, 2013}{}
}

\begin{document}
\begin{center}
{\bf\Large A bijection for tri-cellular maps}
\\
\vspace{15pt} Hillary S.~W. Han and Christian M. Reidys$^{\,\star}$
\end{center}

\begin{center}
         Department of Mathematics and Computer Science  \\
         University of Southern Denmark, Campusvej 55, \\
         DK-5230, Odense M, Denmark \\
         Phone$^{\,\star}$: 45-24409251 \\
         Fax$^{\,\star}$: 45-65502325 \\
         email$^{\,\star}$: duck@santafe.edu
\end{center}

\begin{abstract}
In this paper we give a bijective proof for a relation between uni- bi- and tricellular maps
of certain topological genus. While this relation can formally be obtained using Matrix-theory
as a result of the Schwinger-Dyson equation, we here present a bijection for the
corresponding coefficient equation. Our construction is facilitated by repeated application of a
certain cutting, the contraction of edges, incident to two vertices and the deletion
of certain edges.
\end{abstract}

\section{Introduction}

$k$-cellular maps can be viewed as drawings on a topological surface, they represent a
cell-complex of the latter and inherit the topological genus of the surface as their
geometric realization.

In a seminal paper Harer and Zagier \cite{Zagier86} computed the virtual Euler characteristic
of the Moduli space of curves, independently derived by Penner
\cite{Penner7} and still lack a combinatorial interpretation.
Key object here play unicellular maps \cite{Chapuy:11} of genus $g$ with $n$ edges, $U_g(n)$,
i.e.~fatgraphs\cite{Penner-waterman, Penner87,L-M}
with a unique boundary component. Most prominantly here is the recursion
$$
(n+1) {\sf u}_{g}(n) = 2(2n-1) {\sf u}_g(n-1)+(2n-1)(n-1)(2n-3) {\sf u}_{g-1}(n-2)
$$
In \cite{Zagier, Reidys:top1} the generating function of unicellular maps is obtained as
$$
{\bf U}_g(z) = \frac{P_g(z)}{(1-4z)^{3g-1/2}},
$$
where $P_g(z)$ is polynomial defined over the integers of degree at most $3g-1$ that is
divisible by $z^{2g}$ with $P_g(1/4) \neq 0$,
$[z^{2g}]P_g(z) \neq 0$ and $[z^h]P_g(z) = 0$ for $0 \leq h \leq 2g-1$.

Matrix-theory \cite{Dyson,Schwinger}, via the Schwinger-Dyson equation or representation
theorey \cite{Sagan}, connects the
generating functions of unicellular, ${\bf U}_{g}(z)$, and bicellular maps,
${\bf B}_{g}(z)$. The latter counts fatgraphs having two boundary components that are
connected as combinatorial graphs. The relation can also be proved using the representation
theoretic framework of Zagier \cite{Zagier} and is given by
\begin{equation}\label{E:1}
\sum_{g_1=0}^{g+1}\,{\bf U}_{g_1}(z){\bf U}_{g+1-g_1}(z)
+{\bf B}_{g}(z)= {\bf U}_{g+1}(z)/z.
\end{equation}
Recently \cite{Reidys-Han} the authors presented a bijective proof of the corresponding
coefficient equation
\begin{equation}\label{E:recursion}
\sum_{g_1=0}^{g+1}\,\sum_{i \geq 0}^n\,{\sf u}_{g_1}(i) {\sf u}_{g+1-g_1}(n-i)
+ {\sf b}_g(n)={\sf u}_{g+1}(n+1),
\end{equation}
which revealed a simple construction mechanism. The bijective proof can for instance be
applied, to significantly speed up the folding of RNA interaction structures
\cite{Reidys-Li,fenix2bb}.

An analogous relation between unicellular, bicellular and tricellular maps can also
be obtained via Matrix-theory.
In this paper we give a bijective proof of this relation which reads
\begin{equation}
\begin{split}
& {\sf u}_{g+2}(n+2)=\\
& {\sf t}_{g}^{}(n)+{\sf d}_{g+2}(n)+4{\sf u}_{g+2}(n+1)-
3{\sf u}_{g+2}(n)+(n+1)(2n+1){\sf u}_{g+1}(n),
\end{split}
\end{equation}
where ${\sf d}_{g+2}(n)$ is explicitly expressed via numbers of unicellular and bicellular maps.

Our strategy is to derive a partition of the set of unicellular maps of
genus $(g+2)$ with $(n+2)$ edges, see Fig.~\ref{F:visualpar} for a first step
of how to decompose the latter.

\begin{figure}[ht]
\begin{center}
\includegraphics[width=0.8\columnwidth]{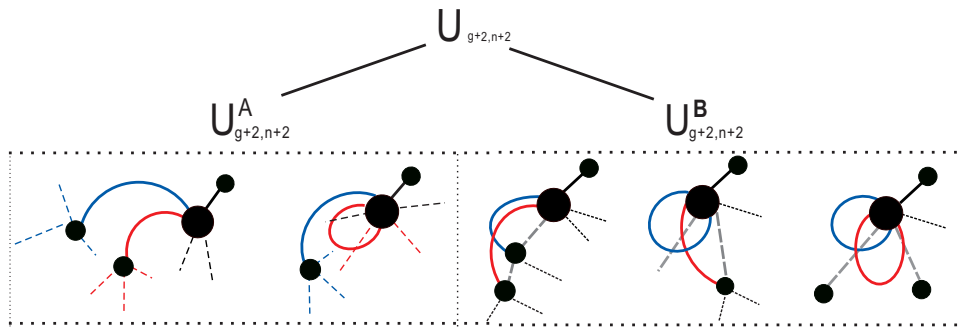}
\end{center}
\caption{\small The first step of partition of $U_{g+2, n+2}$.}
\label{F:visualpar}
\end{figure}

It is interesting to note that Matrix-theory does not provide any insight
w.r.t.~for instance quadricellular maps. It seems in fact unlikely that such relations
can be derived using this formal framework. The bijective proof presented here however
is rather straightforward once the correct partitioning is identified. We believe that
it is very well possible to prove similar relations for cellular maps with more than
three boundary components.

%%%
%%%%%%%%%%%%%%%%%%%%%%%%%%%%%%%%%%%%%%%%%%%%%%%%%%%%%%%%%%%%%%%%%%%%%%%%%%%%%%%%%
%%%
\section{Basic Definitions}
%%%
%%%%%%%%%%%%%%%%%%%%%%%%%%%%%%%%%%%%%%%%%%%%%%%%%%%%%%%%%%%%%%%%%%%%%%%%%%%%%%%%%
%%%
Let $S_{2n}$ denote the permutation group over $2n$ elements.
%%%
%%%%%%%%%%%%%%%%%%%%%%%%%%%%%%%%%%%%%%%%%%%%%%%%%%%%%%%%%%%%%%%%%%%%%%%%%%%%%%%%%
%%%
\begin{definition}
Let $k,J$ be positive integers. A $k$-cellular map is a triple
$(H, \alpha, (\gamma_i)_{1\le i\le k})$, where $H$ is
a set of cardinality $2n$, $\alpha$ a fixed-point free
involution and $\gamma_i$ are cycles such that
$\gamma=\prod^k_{i=1} \gamma_i \in S_{2n}$.
The elements of $H$ are half-edges, the cycles of $\alpha$ are
edges. The cycles of the permutation $\sigma=\alpha \circ
\gamma$ are the vertices $v_i$, $1 \leq i \leq J$.
The length of $v_i$ is its degree. The cycle $\gamma_i$ is
the $i$-th face.
\end{definition}

%%%
%%%%%%%%%%%%%%%%%%%%%%%%%%%%%%%%%%%%%%%%%%%%%%%%%%%%%%%%%%%%%%%%%%%%%%%%%%%%%%%%%
%%%
The combinatorial graph $G$ of a $k$-cellular map is the graph whose edges and
vertices are the cycles of $\alpha$ and $\sigma$.
We can regard a $G$-edge as a ribbon whose two sides are labeled by the
half-edges as follows: each side of the ribbon represents one half-edge, we
decide which half-edge corresponds to which side of the ribbon by the
convention that, if a half-edge $h$ belongs to a cycle $e$ of $\alpha$ and
a certain $v$ of $\sigma$, then $h$ is the right-hand side of the ribbon
corresponding to $e$, when entering $v$. Furthermore, around each vertex
$v$, the counterclockwise ordering of the half-edges belonging to the cycle
$v$ is given by that cycle, we obtain a graphical object called the fatgraph
, $\mathbb{G}$,
tantamount to $(H, \alpha, (\gamma_i)_i)$ and the graph $G$ is the corresponding
combinatorial graph of $\mathbb{G}$.
%%%
%%%%%%%%%%%%%%%%%%%%%%%%%%%%%%%%%%%%%%%%%%%%%%%%%%%%%%%%%%%%%%%%%%%%%%%%%%%%%%%%%
%%%
\begin{definition}
A planted $k$-cellular map is a $k$-cellular map in which each $\gamma_i$
contains a distinguished half-edge $p_i$, such that $(p_i)$ is a
$\sigma$-cycle. $(p_i)$ is called the plant of the face $\gamma_i$ and
$\sigma$-cycles, except of the plants are called $(np)$-vertices.
\end{definition}
%%%
%%%%%%%%%%%%%%%%%%%%%%%%%%%%%%%%%%%%%%%%%%%%%%%%%%%%%%%%%%%%%%%%%%%%%%%%%%%%%%%%%
%%%

In the following, we refer to edges not incident to plants as $(np)$-edges.
Let $X_{k}(n)$ denote the set of planted $k$-cellular maps that contain $n$
$(np)$-edges.

In planted maps we shall label the half-edges of $H$ such that
$(\alpha(p_i), p_i)=(R_i, S_i)$, that is
\begin{equation}\label{E:relabel}
\begin{split}
&\gamma_1=(R_1, 1, 2, \ldots, m_1, S_1), \\
&\gamma_i=(R_i, m_{i-1}+1, m_{i-1}+2, \ldots, m_{i}-1, m_{i}, S_i),  \quad 2\leq i \leq k.
\end{split}
\end{equation}
Given $x_{k, n} \in X_k(n)$ we define the linear order
$<_{\gamma}$ on $H$ for each face $\gamma_i$ via:
$$
S_{i-1}<_{\gamma} R_i <_{\gamma} \gamma_i(R_i)<_{\gamma} \gamma_i^2(R_i)<_{\gamma} \ldots
<_{\gamma} \gamma_i^{m_i}(R_i)<_{\gamma} \gamma_i^{m_i+1}(R_i)=S_i.
$$
Let furthermore $H_{\gamma_1,\dots,\gamma_r}$ denote the set consisting the half-edges in
one of these $\gamma_i$. In particular, $H_{\gamma_i}$ is the set of half-edges
contained in the face $\gamma_i$.

There is a natural equivalence relation over half-edges, $h \sim \alpha(h)$ and
in particular, $\alpha(p_i) \sim p_i$. If $h,\alpha(h) \in H_{\gamma_i}$, then
$(h, \alpha(h))$ is called a one-sided edge and $(h, \alpha(h))$ is called a
two-sided edge, otherwise.

For each vertex $v_j$, let $\min_x(v_j)$ denote the first half-edge via which
$\gamma_i$ enters $v_j$. This gives a canonical way of writing the cycle,
starting at $h_j^1=\min(v_j)$ namely $v_j=(h_j^1, \ldots, h_j^{n_j})$.
In particular, the vertex containing the half-edge $R_1$ is $v_1$,
the ``first'' vertex.

%%%
%%%%%%%%%%%%%%%%%%%%%%%%%%%%%%%%%%%%%%%%%%%%%%%%%%%%%%%%%%%%%%%%%%%%%%%%%%%%%%%%%%%%%%%%%%%
%%%
\section{The partition}
%%%
%%%%%%%%%%%%%%%%%%%%%%%%%%%%%%%%%%%%%%%%%%%%%%%%%%%%%%%%%%%%%%%%%%%%%%%%%%%%%%%%%%%%%%%%%%%%
%%%

$1$-cellular maps are also called unicellular maps \cite{Chapuy:11}.
Let $U_{g, n}$ denote the set of planted, unicellular maps of genus $g$, having
$n$ $(np)$-edges. In particular, let $\epsilon$
denote the unicellular map of genus zero, containing no $(np)$-edge. This map
contains only one edge, the plant, and one additional $(np)$-vertex.

Let $u_{g+2, n+2}=(H, \alpha, \gamma)\in U_{g+2, n+2}$ with face
$\gamma=[R_1, 1, 2, \ldots, 2n+1, 2n+2, S_1]$.
Then
\begin{equation}\label{E:v1}
v_1=(h_1^1, h_1^2, h_1^3 \ldots, h_1^m),  \quad \text{\rm for some } m>0,
\end{equation}
where $h_1^1=R_1$. Thus $\alpha(h_1^2)=1$ and
$\gamma=[R_1, \alpha(h_1^2), 2, \ldots, 2n+1, 2n+2, S_1]$.
In the following we shall identify a partition of $U_{g+2, n+2}$ that will
facilitate our main bijection in Theorem~\ref{T:result}.

To begin, we consider for $m \geq 3$ the four half edges $h_1^2$,
$\alpha(h_1^2)$, $h_1^3$ and $\alpha(h_1^3)$.
Clearly, $\alpha \circ \sigma(h_1^2)=\alpha(h_1^3)$, whence
$h_1^2<_{\gamma} \alpha(h_1^3)$. Furthermore, by construction,
$$
\alpha(h_1^2) <_{\gamma}  h_1^2, \quad
\alpha(h_1^2)  <_{\gamma}  h_1^3,\quad
h_1^2          <_{\gamma}  \alpha(h_1^3),
$$
see also Fig.~\ref{F:branch}.
%%%
%%%%%%%%%%%%%%%%%%%%%%%%%%%%%%%%%%%%%%%%%%%%%%%%%%%%%%%%%%%%%%%%%%%%%%%%
%%%
%%%%%%%%%%%%%%%%%%%%%%%%%%%%%%%%%%%%%%%%%%%%%%%%%%%%%%%%%%%%%%%%%%%%%%%%
%%%
Accordingly, there are the two scenarios
$$
(A) \quad \alpha(h_1^2)<_{\gamma} h_1^2 <_{\gamma} \alpha(h_1^3) <_{\gamma} h_1^3
\quad\text{\rm and}\quad
(B)\quad \alpha(h_1^2)<_{\gamma} h_1^3 <_{\gamma} h_1^2 <_{\gamma} \alpha(h_1^3).
$$
The case $m=2$ belongs to scenario $(A)$, which then reduces to
$$
\quad \alpha(h_1^2)<_{\gamma} h_1^2.
$$
This generates the bipartition of $U_{g+2, n+2}$,
\begin{equation}\label{E:AB}
U_{g+2, n+2}=U_{g+2, n+2}^{A} \, \dot\cup \, U^{B}_{g+2, n+2}.
\end{equation}
%%%
%%%%%%%%%%%%%%%%%%%%%%%%%%%%%%%%%%%%%%%%%%%%%%%%%%%%%%%%%%%%%%%%%%%%%%%%%%%
%%%
\begin{lemma}
In $U_{g+2, n+2}^{A}$-elements the half-edges $\alpha(h_1^2)$ and
$h_1^2$ belong to two different vertices, $v_1$ and $v_2$.
\end{lemma}
%%%
%%%%%%%%%%%%%%%%%%%%%%%%%%%%%%%%%%%%%%%%%%%%%%%%%%%%%%%%%%%%%%%%%%%%%%%%%%%
%%%
\begin{proof}
We have
$$
\alpha(h_1^2) <_{\gamma} h_1^2 <_{\gamma} \alpha(h_1^3) <_{\gamma} h_1^3,
$$
and $\gamma(\alpha(h_1^2))=h_1^2$.
Suppose now $\alpha(h_1^2)$ and $h_1^2$ belong to $v_1$.
Then there exists a half-edge $k_i$ satisfying $\gamma(k_i)=h_1^2$ such that
$h_1^3 <_{\gamma} k_i$ or $h_1^3=k_i$, but this implies $h_1^3 <_{\gamma} h_1^2$,
a contradiction.
\end{proof}

We next refine $U^{A}_{g+2, n+2}$: for $u_{g+2, n+2} \in U_{g+2, n+2}^A$,
we consider the cycle
$$
\bar{\gamma}=
(\alpha(h_1^2), 2, \ldots, h_1^2, \alpha(h_1^3), \ldots, h_1^3, h_1^3+1, \ldots, 2(n+1))
$$
and we use $(\alpha(h_1^i), h_1^i), i=2,3$ to split the $\bar{\gamma}$ into
\begin{equation}\label{E:3GHs}
\bar{\gamma}_{1}=(\alpha(h_1^2), 2,  \ldots,  h_1^2); \quad
\bar{\gamma}_{2}=(\alpha(h_1^3),  \ldots, h_1^3); \quad
\bar{\gamma}_{3} =(h_1^3+1, \ldots, 2(n+1)),
\end{equation}
see Fig.~\ref{F:branch}.

\begin{figure}[ht]
\begin{center}
\includegraphics[width=0.8\columnwidth]{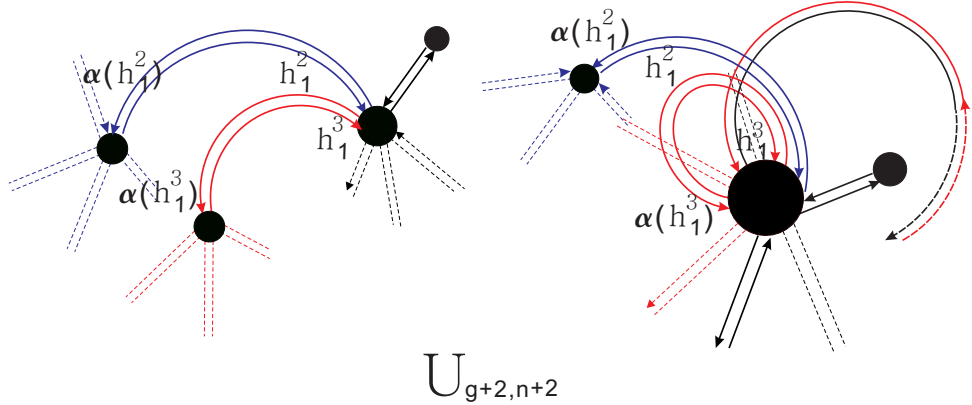}
\end{center}
\caption{\small The three branches (red, blue, black) together with
the two pairs $(\alpha(h_1^2), h_1^2)$ and $(\alpha(h_1^3), h_1^3)$. }
\label{F:branch}
\end{figure}

Suppose the restriction $\alpha|_{S}$ is a welldefined fixed-point free
involution, then we call $S$ closed. Similarly, the sets $H_{\gamma_1,\dots,\gamma_r}$
and $H_{\bar{\gamma}_{i}}$, $i=1, 2, 3$ are called closed, if
$\alpha|_{H_{\gamma_1,\dots,\gamma_r}}$
and $\alpha|_{H_{\bar{\gamma}_{i}}}$ are fixed-point free involutions.

Let $U_{g+2, n+2}^{II}$ denote the subset of
$U^{A}_{g+2, n+2}$-elements in which no $H_{\bar{\gamma}_{i}}$ is closed and let
$U_{g+2, n+2}^{I}$ denote its complement.
Then
\begin{equation}\label{E:I-II}
U^A_{g+2, n+2}=U_{g+2, n+2}^{I}\, \dot\cup \, U_{g+2, n+2}^{II}.
\end{equation}
We refine $U_{g+2,n+2}^{I}$ further:
\begin{itemize}
\item $U^1_{g+2, n+2}$: the set of $U_{g+2, n+2}^{I}$-elements in which exactly two
      ${\bar{\gamma}_{i}}, {\bar{\gamma}_{j}}$ are empty,
\item $U^2_{g+2, n+2}$: the set of $U_{g+2, n+2}^{I}$-elements in which exactly
      one ${\bar{\gamma}_{i}}$ is empty,
\item $U^{> 2}_{g+2, n+2}$: the complement of $U^1_{g+2, n+2}$ and $U^2_{g+2, n+2}$,
      that is, the set of $U_{g+2, n+2}^{I}$ in which no ${\bar{\gamma}_{i}}$ is empty.
\end{itemize}
Thus
\begin{equation}\label{F:disU}
U^{I}_{g+2, n+2}= U^1_{g+2, n+2} \,\dot\cup\,  U^{2}_{g+2, n+2}\,\dot\cup \, U^{>2}_{g+2, n+2}.
\end{equation}

We refine $U_{g+2,n+2}^{>2}$ a bit more, for this purpose let
\begin{itemize}
\item $U^{>2, 3}_{g+2, n+2}$: be the subset of $U_{g+2, n+2}^{>2}$-elements in which
      $\bar{\gamma}_{1}=(\alpha(h_1^2), h_1^2)$.
\item $U^{>2, 4}_{g+2, n+2}$: be the subset of $U_{g+2, n+2}^{>2}$-elements in which
      $\bar{\gamma}_{1}=(\alpha(h_1^2), k^2_1, \ldots, k_1^m)$, $m \geq 4$ and
      $\bar{\gamma}_{2}=(\alpha(h_1^3), h_1^3)$.
\item $U^{5}_{g+2, n+2}$: the complement of $U^{>2, 3}_{g+2, n+2}$ and $U^{>2, 4}_{g+2, n+2}$,
      that is subset of $U^{>2}_{g+2, n+2}$-elements in which
      $\bar{\gamma}_{1}=(\alpha(h_1^2), k^2_1, \ldots, k_1^m)$ and
      $\bar{\gamma}_{2}=(\alpha(h_1^3), k^2_2, \ldots, k_2^l)$, $m,l \geq 4$.
\end{itemize}

Accordingly,
\begin{equation}\label{E:geq2}
U^{>2}_{g+2, n+2}=U^5_{g+2, n+2} \,\dot\cup \, U^{>2,3}_{g+2, n+2}\,\dot\cup \, U^{>2,4}_{g+2, n+2}.
\end{equation}

Furthermore we present
$U^{>2,3}_{g+2, n+2}$ :
\begin{equation}\label{E:U-g2-3}
U^{>2,3}_{g+2, n+2}=U^{3}_{g+2, n+2}\setminus  U^{m, 1}_{g+2, n+2},
\end{equation}
where
\begin{itemize}
\item $U^{3}_{g+2, n+2}$ denotes the subset of $U_{g+2, n+2}^{I}$-elements in
which $\bar{\gamma}_{1}=(\alpha(h_1^2), h_1^2)$,
\item $U^{m, 1}_{g+2, n+2}$ denotes the subset of $U^2_{g+2, n+2}$-elements
      in which $\bar{\gamma}_{1}=(\alpha(h_1^2), h_1^2)$.
\end{itemize}
Furthermore we present $U^{>2,4}_{g+2,n+2}$ as
\begin{equation}\label{E:U-g3-4}
U^{>2,4}_{g+2, n+2}=U^{4}_{g+2, n+2}\setminus (U^{m, 2}_{g+2, n+2}\,\dot\cup\, U^{m, 3}_{g+2, n+2}),
\end{equation}
where
\begin{itemize}
\item $U^{4}_{g+2, n+2}$ is the subset of $U_{g+2, n+2}^{I}$-elements
      with $\bar{\gamma}_{2}=(\alpha(h_1^3), h_1^3)$,
\item $U^{m, 2}_{g+2, n+2}$ is the subset of $U^2_{g+2, n+2}$-elements
      with $\bar{\gamma}_{2}=(\alpha(h_1^3), h_1^3)$,
\item $U^{m, 3}_{g+2, n+2}$ is the subset of $U^{>2}_{g+2, n+2}$-elements
      with $\bar{\gamma}_{1}=(\alpha(h_1^2), h_1^2)$ and
      $\bar{\gamma}_{2}=(\alpha(h_1^3), h_1^3)$.
\end{itemize}

%%%
%%%%%%%%%%%%%%%%%%%%%%%%%%%%%%%%%%%%%%%%%%%%%%%%%%%%%%%%%%%%%%%%%%%%%%%%%%%%%%%%%%%%%%%
%%%
\section{Some lemmas}
%%%
%%%%%%%%%%%%%%%%%%%%%%%%%%%%%%%%%%%%%%%%%%%%%%%%%%%%%%%%%%%%%%%%%%%%%%%%%%%%%%%%%%%%%%%
%%%

In this section we state three procedures that are employed repeatedly in our
bijection. They are ``cutting'', ``contraction'' and ``deletion''. These procedures
constitute the key three operations that, applied in various contexts, facilitate
the bijection.

%%%
%%%%%%%%%%%%%%%%%%%%%%%%%%%%%%%%%%%%%%%%%%%%%%%%%%%%%%%%%%%%%%%%%%%%%%%%%%%%%%%%%%%%%%%
%%%
\begin{lemma}\label{L:alpha-cutting}{\bf  (Cutting)}
Suppose we are given a planted, unicellular map $u=(H, \alpha, \gamma) \in
U^A_{g+2, n+2}$ with
\begin{equation}\label{E:face}
\gamma=(R_1, \alpha(h_1^{2}), \ldots, h_1^{2},
\alpha(h_1^{3}), \ldots, h_1^{3}, h_1^{3}+1, \ldots, 2(n+1), S_1).
\end{equation}
Then $u$ can be mapped to a planted, $3$-cellular map, $x_{3, n+2}\in X_{3}(n+2)$, with
the three faces $\gamma_1, \gamma_2, \gamma_3$ via

\begin{equation}\label{E:c2}
\begin{split}
& c_2\colon U^A_{g+2, n+2}  \longrightarrow X_{3}(n+2), \\
& (H,\alpha,\gamma)\mapsto
(H,\alpha,(\gamma_1, \gamma_2, \gamma_3))
\end{split}
\end{equation}

where
\begin{equation}\label{E:gamma123}
\gamma_1=(R_1, 1, \ldots, m_1, S_1), \
\gamma_2=(R_2, m_1+1, \ldots, m_2, S_2), \
\gamma_3=(R_3, m_2+1, \ldots, m_3, S_3).
\end{equation}
Furthermore, the mapping $c_2$ has the following inverse:
\begin{equation}\label{E:g2}
\begin{split}
& g_2 \colon X_{3}(n+2) \longrightarrow U^A_{g+2, n+2}\\
& (H,\alpha, (\gamma_1, \gamma_2, \gamma_3)) \mapsto  (H,\alpha,\gamma).
\end{split}
\end{equation}
\end{lemma}
%%%
%%%%%%%%%%%%%%%%%%%%%%%%%%%%%%%%%%%%%%%%%%%%%%%%%%%%%%%%%%%%%%%%%%%%%%%%%%%%%%%%%%%%%%%
%%%
\begin{proof}
By assumption we have
$$
\alpha(h_1^2)<_{\gamma}h_1^2 <_{\gamma} \alpha(h_1^3) <_{\gamma} h_1^3,
$$
whence the face of $u_{g+2, n+2}$ can be written as in eq.~(\ref{E:face}).
We use $(\alpha(h_1^i), h_1^i), i=2,3$ and $\bar{\gamma}_{i}, i=1, 2, 3$ which are given by  ~\ref{E:3GHs},
then concatenate the sequence of half-edges of $(R_1)$, $\bar{\gamma}_{3}$ and $(S_1)$ to form
\begin{equation}\label{E:h-bc}
\begin{split}
&\gamma_1=\bar{\gamma}_{1}, \quad
\gamma_2=\bar{\gamma}_{2}, \quad\\
&\gamma_3=(R_1, h_1^{3}+1, \ldots, 2(n+1), S_1)
\end{split}
\end{equation}
and relabel the cycles as in eq.~(\ref{E:gamma123}).
This produces the plants $(S_1)$, $(S_2)$ and $(S_3)$.
Since $\prod_{i=1}^3 \gamma_i\in S_{2n+2}$, $c_2(H, \alpha, \gamma)$
is a $3$-cellular map, $c_2$ is well-defined, see Fig.~\ref{F:a-cutting}.

\begin{figure}[ht]
\begin{center}
\includegraphics[width=0.7\columnwidth]{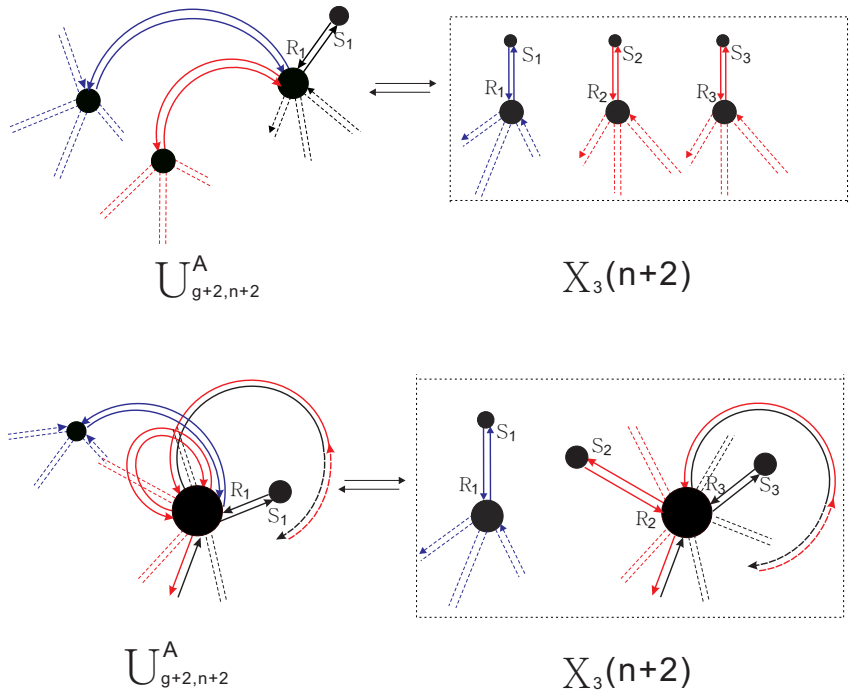}
\end{center}
\caption{\small The mappings $c_2$ and $g_2$.}\label{F:a-cutting}
\end{figure}

We next construct an explicit inverse of $c_2$. Suppose we are given a
$3$-cellular map $c_2((H,\alpha,\gamma))$, in which the $\gamma_i$ are
as in eq.~(\ref{E:h-bc}). Then we concatenate the sequences of
half-edges of the three $\gamma_i$-cycles and relabel as in
eq.~(\ref{E:face}), i.e.~$\gamma(h_1^2)=\alpha(h_1^3)$,
$\gamma(h_1^3)=h_1^3+1$ and $\gamma(R_1)=h_1^3+1$.
We derive, by construction,
$$
\alpha(h_1^2)<_{\gamma}h_1^2 <_{\gamma} \alpha(h_1^3) <_{\gamma} h_1^3.
$$
Accordingly, $g_2(c_2(H, \alpha, \gamma))=(H, \alpha,\gamma)$
is a unicellular map of genus $(g+2)$ with property $(A)$.
\end{proof}

%%%
%%%%%%%%%%%%%%%%%%%%%%%%%%%%%%%%%%%%%%%%%%%%%%%%%%%%%%%%%%%%%%%%%%%%%%%%%%%%
%%%%
\begin{lemma}\label{L:contract2} {\bf (Contraction)}
Suppose $u \in U_{g+2, n+2}$ has a one-sided edge
$(\alpha(h_1^{l_1}), h_1^{l_1})$, $\alpha(h_1^{l_1})<_{\gamma} h_1^{l_1}$,
such that $\alpha(h_1^{l_1})$ and $h_1^{l_1}$ are incident to two
different vertices $v_j,v_1$. Relabeling the two half-edges
we can write the face
\begin{equation}\label{E:bc-e2}
\gamma=(R_1, K_1, \alpha(h_1^{l_1}), K_2, h_1^{l_1}, K_3, S_1).
\end{equation}
Here either $K_1=k_1^1, \ldots, k_1^{n_1}$ or $K_1=\varnothing$,
$K_2=k_2^1, \ldots, k_2^{n_2}$ or $K_2=\varnothing$ and either
$K_3=k_3^1, \ldots, k_3^{n_3}$ or $K_3=\varnothing$.
Then $u$ corresponds to a unicellular map $u'$ together with two
distinguished half-edges via mapping
\begin{equation}\label{E:m_2}
\begin{split}
&m_2\colon U_{g+2, n+2}  \longrightarrow U_{g+2, n+1}, \\
&((H, \alpha,\gamma), (\alpha(h_1^{l_1}), h_1^{l_1})) \mapsto
((H',\alpha',\gamma'), (\alpha(h_1^{l_1})-1, h_1^{l_1}-1))
\end{split}
\end{equation}
where $H'=H \setminus\{h_1^{l_1}, \alpha(h_1^{l_1})\}$,
$\alpha'=\alpha \setminus (h_1^{l_1}, \alpha(h_1^{l_1})$,
$\gamma'=(R_1, K_1, K_2, K_3, S_1)$ and
$\alpha(h_1^{l_1})-1=k_1^{n_1}, h_1^{l_1}-1=k_2^{n_2}$, if
$K_1,K_2 \neq \varnothing$,
$\alpha(h_1^{l_1})-1=h_1^{l_1}-1=k_1^{n_1}$, if $K_1 \neq \varnothing$ and
$K_2 =\varnothing$,
$\alpha(h_1^{l_1})-1=R_1$,  $h_1^{l_1}-1=k_2^{n_1}$, if $K_1 = \varnothing$ and
$K_2 \neq \varnothing$,
and finally
$\alpha(h_1^{l_1})=h_1^{l_1}=R_1$, if $K_1,K_2 = \varnothing$.\\
Furthermore the mapping $e_2$
\begin{equation}\label{E:e_2}
\begin{split}
& e_2 \colon U_{g+2, n+1}  \longrightarrow U_{g+2, n+2}, \\
& ((H',\alpha',\gamma'), (\alpha(h_1^{l_1})-1, h_1^{l_1}-1))\mapsto
((H,\alpha,\gamma), (\alpha(h_1^{l_1}), h_1^{l_1}))
\end{split}
\end{equation}
has the property $e_2\circ m_2=\text{\rm id}$.
\end{lemma}
%%%%
%%%%%%%%%%%%%%%%%%%%%%%%%%%%%%%%%%%%%%%%%%%%%%%%%%%%%%%%%%%%%%%%%%%%%%%%%%%%%%%%%
%%%%
\begin{proof}
$m_2((H,\alpha,\gamma), (\alpha(h_1^{l_1}), h_1^{l_1}))$ is by construction unicellular
and retains the genus of $(H,\alpha,\gamma)$.
\end{proof}
We describe the contraction in Fig.~\ref{F:contract}.
\begin{figure}[ht]
\begin{center}
\includegraphics[width=0.6\columnwidth]{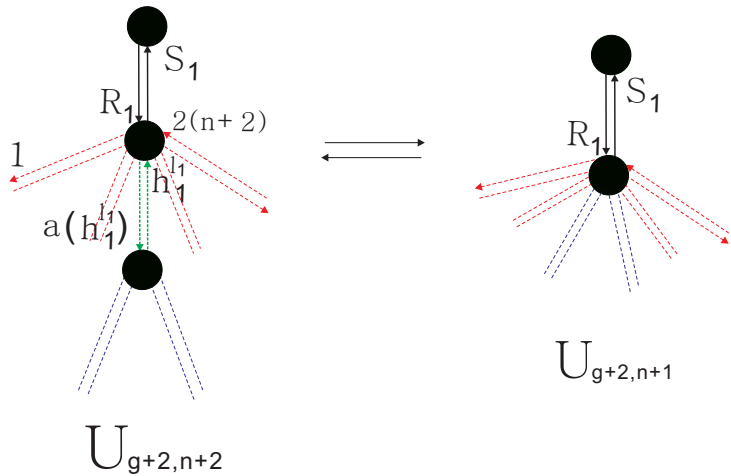}
\end{center}
\caption{\small The mappings $m_2$ and $e_2$. The edge $(\alpha(h_1^{l_1}), h_1^{l_1})$
(green) is one-sided edge.}\label{F:contract}
\end{figure}

%%%
%%%%%%%%%%%%%%%%%%%%%%%%%%%%%%%%%%%%%%%%%%%%%%%%%%%%%%%%%%%%%%%%%%%%%%%%%%%%%%%%%%%%%%%
%%%
\begin{lemma}\label{L:beta-cutting}{\bf (Deletion)}
Given a unicellular map $u=(H,\alpha,\gamma)\in U_{g+2,n+2}^B$ with face
\begin{equation}\label{E:delete-face1}
\gamma=(R_1, \alpha(h_1^2), K_1, h_1^3, K_2, h_1^2, \alpha(h_1^3), K_3, S_1),
\end{equation}
where $K_1=k_1^1, \ldots, k_1^{n_1}$ or $K_1=\varnothing$,
$K_2=k_2^1, \ldots, k_2^{n_2}$ or $K_2=\varnothing$ and
$K_3=k_3^1, \ldots,k_3^{n_3}$ or $K_3=\varnothing$.\\
Then $u$ corresponds to a unicellular map $u'=(H', \alpha', \gamma') \in
U_{g+1, n}$ together with two half-edges $k_{l_1}$ and $k_{l_2}$, where
$k_{l_1} \leq_{\gamma} k_{l_2}$, via the mapping
\begin{equation}\label{E:r2}
\begin{split}
& r_2 \colon U^B_{g+2, n+2} \longrightarrow U_{g+1, n}\\
& ((H,\alpha,\gamma)) \mapsto ((H',\alpha',\gamma'), (k_{l_1}, k_{l_2})),
\end{split}
\end{equation}
where $H'=H\setminus \{\alpha(h_1^2), h_1^2, \alpha(h_1^3), h_1^3 \}$,
$\alpha'=\alpha \setminus \{ (h_1^2, \alpha(h_1^2)),
(h_1^3, \alpha(h_1^3))\}$ and
\begin{equation}\label{E:newbc}
\gamma'=(R_1, K_2, K_1, K_3, S_1).
\end{equation}
$r_2$ can be reversed by mapping a unicellular map
$u=(H, \alpha', \gamma')$, together with two arbitrary half-edges
$k_{l_1}$ and $k_{l_2}$ ($k_{l_1} <_{\gamma} k_{l_2}$) as follows:
\begin{equation}\label{E:s2}
\begin{split}
& s_2 \colon U_{g+1, n} \longrightarrow  U^B_{g+2, n+2}   \\
&((H',\alpha',\gamma'), (k_{l_1}, k_{l_2})) \mapsto (H,\alpha,\gamma).
\end{split}
\end{equation}
\end{lemma}
%%%
%%%%%%%%%%%%%%%%%%%%%%%%%%%%%%%%%%%%%%%%%%%
%%%
\begin{proof}
By construction $\gamma'\in S_{2(n-2)}$, $\alpha'$ is a fixed-point free
involution and $H'$ has cardinality $2(n-2)$, whence $r_2((H,\alpha,\gamma))$
is unicellular.
Euler characteristic implies that the genus of $r_2((H,\alpha,\gamma))$
is $(g-1)$. Moreover, we have in case of $K_1,K_2 \neq \varnothing$,
$k_{l_1}=k_1^{n_1}, k_{l_2}=k_2^{n_2}$, in case of $K_1 \neq \varnothing$
$K_2 =\varnothing$, $k_{l_1}=k_{l_2}=k_1^{n_1}$, in case of
$K_1 = \varnothing$ and $K_2 \neq \varnothing$, $k_{l_1}=R_1$,
$k_{l_2}=k_2^{n_1}$, and in case of
$K_1,K_2=\varnothing$ we have $k_{l_1}=k_{l_2}=R_1$, see Fig.~\ref{F:typeB}.

\begin{figure}[ht]
\begin{center}
\includegraphics[width=0.7\columnwidth]{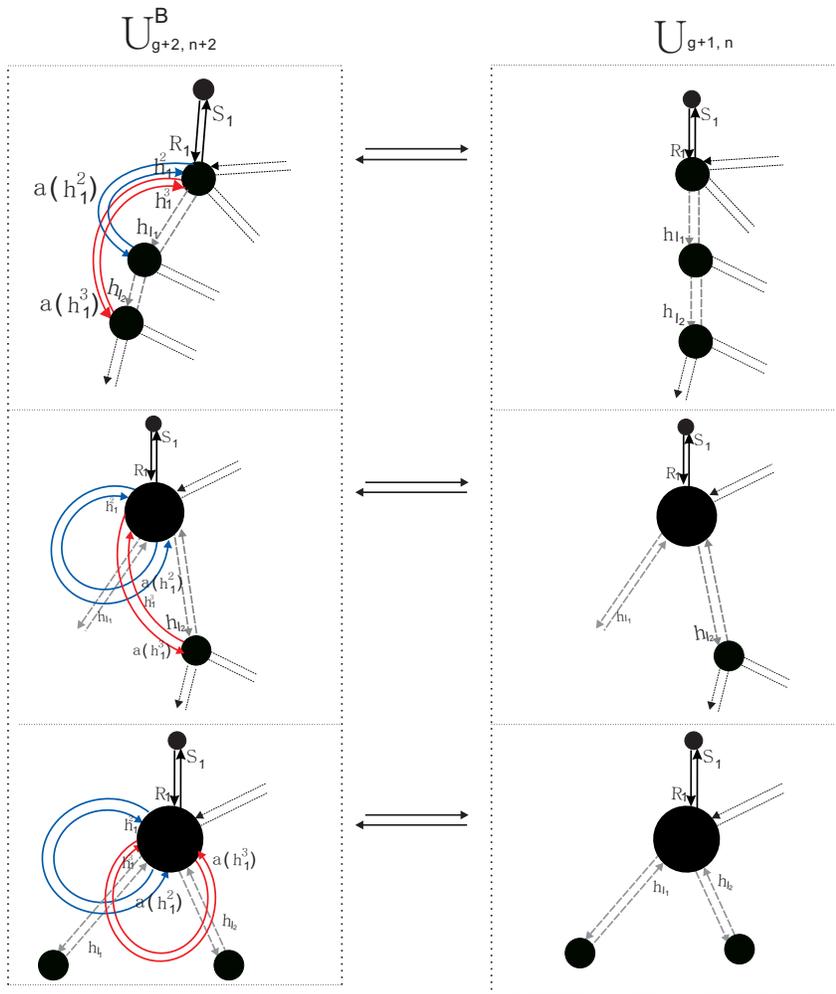}
\end{center}
\caption{\small The mappings $r_2$ and $s_2$. The deleting edges are
$(\alpha(h_1^2), h_1^2)$ (blue) and $(\alpha(h_1^3), h_1^3)$ (red).
The gray dotted lines denote a sequence of half-edges connecting
two vertices.}\label{F:typeB}
\end{figure}
%%%%
%%%%%%%%%%%%%%%%%%%%%%%%%%%%%%%%%%%%%%%%%%%%%%%%
%%%%
Given a unicellular map $(H', \alpha', \gamma') \in U_{g+1, n}$, there are
$$
{2n+1 \choose 2}+{2n+1 \choose 1}=\frac{(2n+1)(2n)}{2}+2n+1
=(2n+1)(n+1)
$$
ways to choose $(k_{l_1}, k_{l_2})$ such that $k_{l_1} <_{\gamma} k_{l_2}$.
We now select two half-edges $(k_{l_1}, k_{l_2})$ such that $k_{l_1} <_{\gamma}
k_{l_2}$ and insert the pairs of half-edges
$(\alpha(h_1^i), h_1^i)$, $i=2,3$ into the face $\gamma'$. This produces
the face $\gamma$, with
$\gamma(R_1)=\alpha(h_1^2)$, $\gamma(k_{l_2})=h_1^2$,
$\gamma(k_{l_1})=h_1^3$ and $\gamma(h_1^2)=\alpha(h_1^3)$,
$\gamma(h_1^3)=\gamma'(R_1)$ and $\gamma(\alpha(h_1^3))=\gamma'(k_{l_2})$.
Consequently we have
$$
\gamma=(R_1, \alpha(h_1^2), \ldots, k_{l_1}, h_1^3, \ldots, k_{l_2}, h_1^2,
\alpha(h_1^3), \ldots,  S_1).
$$
We then relabel $\gamma$ as in eq.~(\ref{E:delete-face1}). Since $\alpha$ is
a fixed-point free involution and $H$ is a set of cardinality $2(n+2)$,
$s_2 (H', \alpha', \sigma')$ is a unicellular map with property
$B$. Euler characteristic implies $s_2(H', \alpha', \sigma')$ has genus $(g+2)$.
By construction, we have $s_2 \circ r_2=id_{U_{g+2, n+2}^{B}}$.
\end{proof}

%%%%
%%%%%%%%%%%%%%%%%%%%%%%%%%%%%%%%%%%%%%%%%%%%%%%%%%%%%%%%%%%%%%%%%%%%%%%%%%%%%
%%%%
\section{The main theorem}
%%%
%%%%%%%%%%%%%%%%%%%%%%%%%%%%%%%%%%%%%%%%%%%%%%%%%%%%%%%%%%%%%%%%%%%%%%%%%%%%%%
%%%
In this section we state some auxiliary bijections and our main result. We furthermore
give in Fig.~\ref{F:algo} an modular description of how our bijection works.

We call a planted $2$-cellular map, whose combinatorial graph is connected,
a planted, bicellular map. Let $B_{g, n}$ denote the set of planted,
bicellular maps of genus $g$ with $n$ $(np)$-edges.

Let $U^{5, i}_{g+2, n+2}$ denote the subset of $U^5_{g+2, n+2}$ in which only a single
$H_{\bar{\gamma}_{i}}$, $i=1,2,3$ is closed and let $U_{g+2,n+2}^{5, 4}$ denote the set of
$U^5_{g+2, n+2}$-elements in which all $H_{\bar{\gamma}_{i}}$, $i=1, 2, 3$ are closed,
i.e.~$U^5_{g+2, n+2}=\dot{\cup}_{i=0}^4 U^{5, i}_{g+2, n+2}$.

%%%
%%%%%%%%%%%%%%%%%%%%%%%%%%%%%%%%%%%%%%%%%%%%%%%%%%%%%%%%%%%%%%%%%%%%%%%%%%%%%%%%%%%%%%%
\begin{lemma}\label{L:disconnect}
We have the bijections
\begin{eqnarray*}
\eta_{5,i}\colon U^{5,i}_{g+2, n+2} & \longrightarrow &
\dot\bigcup_{0\le g_3 \le g+1,\; 0\le j, k\le n}\left( U_{g_3, j}
\times B_{g+1-g_3, n-j}\right), \quad  1\leq i \leq 3;\\
\eta_{5,4}\colon U^{5,4}_{g+2, n+2} & \longrightarrow &  \bigcup_{0\le g_1, g_2\le g+2,\; 0\le j, k\le n}
U_{g_1, j}\times U_{g_2, k} \times U_{g+2-g_1-g_2, n-j-k}.
\end{eqnarray*}
\end{lemma}
%%%
%%%%%%%%%%%%%%%%%%%%%%%%%%%%%%%%%%%%%%%%%%%%%%%%%%%%%%%%%%%%%%%%%%%%%%%%%%%%%%%%%%%%%%%
%%%
We prove Lemma~\ref{L:disconnect} in Section~\ref{S:proofs}.
%%%
%%%%%%%%%%%%%%%%%%%%%%%%%%%%%%%%%%%%%%%%%%%%%%%%%%%%%%%%%%%%%%%%%%%%%%%%%%%%%%%%%%%%%%%
%%%
\begin{lemma}\label{L:add-1edge}
There are fours bijections, $\eta_i$ for $i=1, \ldots, 4$,
$$\eta_{1}\colon  U^1_{g+2, n+2} \longrightarrow  U_{g+2, n+1}.$$
$$\eta_{2}\colon  U^2_{g+2, n+2} \longrightarrow U_{g+2, n+1}.$$
$$\eta_{3}\colon  U^3_{g+2, n+2} \longrightarrow U_{g+2, n+1}.$$
$$\eta_{4}\colon  U^4_{g+2, n+2} \longrightarrow U_{g+2, n+1}.$$
\end{lemma}
%%%
We prove Lemma~\ref{L:add-1edge} in Section~\ref{S:proofs}.

%%%
%%%%%%%%%%%%%%%%%%%%%%%%%%%%%%%%%%%%%%%%%%%%%%%%%%%%%%%%%%%%%%%%%%%%%%%%%%%%%%%%%%%%%%%
%%%
\begin{lemma}\label{L:add-2edge}
We have the three bijections:
$$\eta_{5}\colon  U^{m,1}_{g+2, n+2} \longrightarrow U_{g+2, n}.$$
$$\eta_{6}\colon  U^{m,2}_{g+2, n+2} \longrightarrow U_{g+2, n}.$$
$$\eta_{7}\colon  U^{m,3}_{g+2, n+2} \longrightarrow U_{g+2, n}.$$
\end{lemma}
%%%
%%%%%%%%%%%%%%%%%%%%%%%%%%%%%%%%%%%%%%%%%%%%%%%%%%%%%%%%%%%%%%%%%%%%%%%%%%%%%%%%%%%%%%%
%%%
We prove Lemma~\ref{L:add-2edge} in Section~\ref{S:proofs}.

%%%
%%%%%%%%%%%%%%%%%%%%%%%%%%%%%%%%%%%%%%%%%%%%%%%%%%%%%%%%%%%%%%%%%%%%%%%%%%%%%%%%%%%%%%%
%%%
A planted $3$-cellular map that is connected as a combinatorial graph is
called a planted tri-cellular map. Let $T_{g, n}$ denote the set of planted,
tricellular maps of genus $g$ with $n$ $(np)$-edges.

%%%
%%%%%%%%%%%%%%%%%%%%%%%%%%%%%%%%%%%%%%%%%%%%%%%%%%%%%%%%%%%%%%%%%%%%%%%%%%%%%%%%%%%%%%%
%%%
\begin{proposition}\label{P:tricellular}
There is a bijection
$$
\theta \colon U^{II}_{g+2, n+2} \longrightarrow T_{g, n}.
$$
\end{proposition}
%%%
%%%%%%%%%%%%%%%%%%%%%%%%%%%%%%%%%%%%%%%%%%%%%%%%%%%%%%%%%%%%%%%%%%%%%%%%%%%%%%%%%%%%%%%
%%%
We prove Proposition~\ref{P:tricellular} in Section~\ref{S:proofs}.
%%%
%%%%%%%%%%%%%%%%%%%%%%%%%%%%%%%%%%%%%%%%%%%%%%%%%%%%%%%%%%%%%%%%%%%%%%%%%%%%%%%%%%%
%%%
\begin{proposition}\label{P:insert2edges}
There is a bijection
$$
\psi \colon U^{B}_{g+2, n+2} \longrightarrow (2n+1)(n+1)U_{g+1, n}.
$$
\end{proposition}
%%%
%%%%%%%%%%%%%%%%%%%%%%%%%%%%%%%%%%%%%%%%%%%%%%%%%%%%%%%%%%%%%%%%%%%%%%%%%%%%%%%%%%%
%%%
We prove Proposition~\ref{P:tricellular} in Section~\ref{S:proofs}.
In Figure~\ref{F:algo} we give an overview of how the above bijections are applied.

\begin{figure}
\begin{center}
\includegraphics[width=0.9\columnwidth]{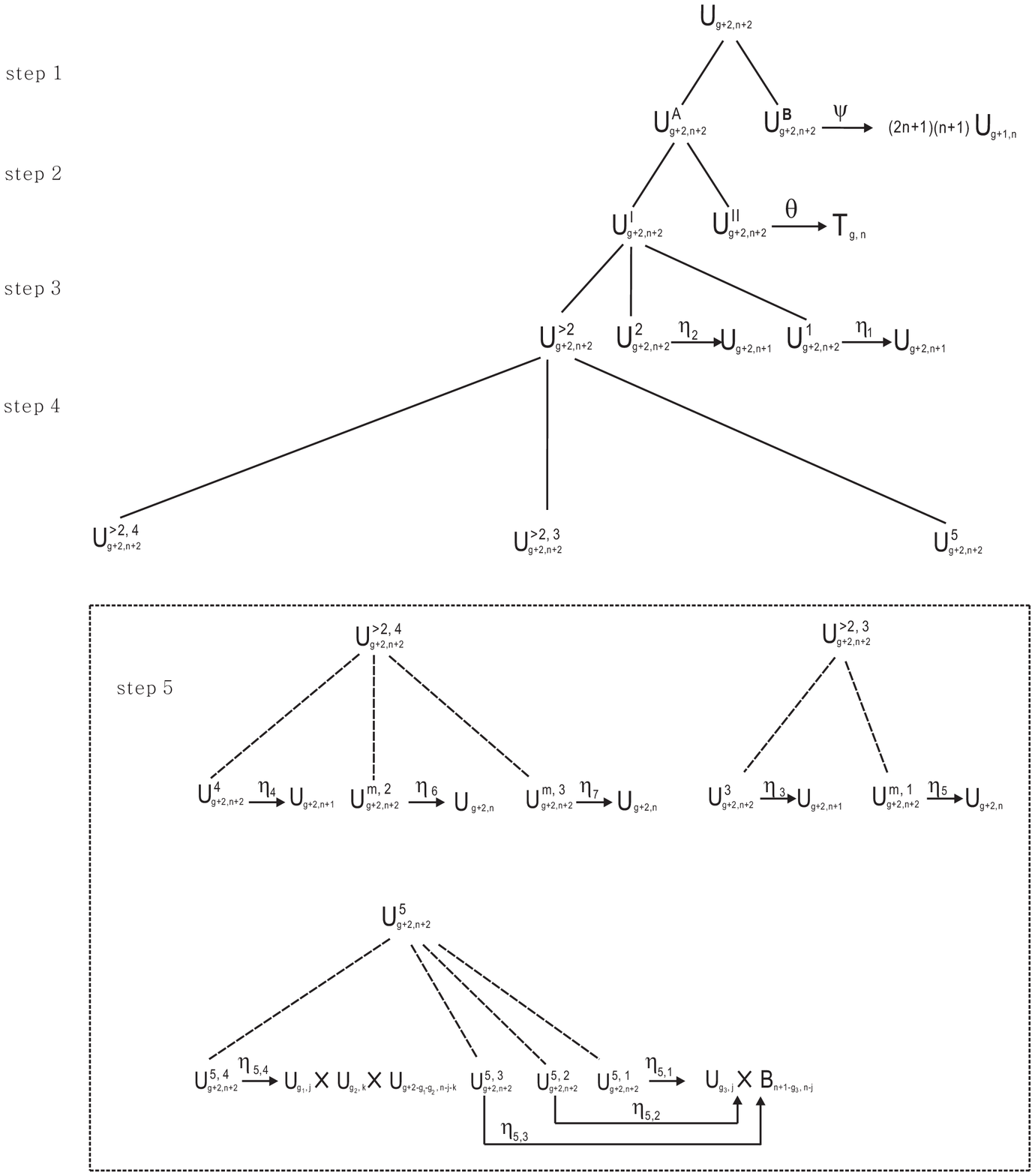}
\end{center}
\caption{\small Applying the bijections.}
\label{F:algo}
\end{figure}

For a set $A^{\xi}_{g, n}$ we denote its cardinality by ${\sf a}^{\xi}_{g}(n)$.

%%%
%%%%%%%%%%%%%%%%%%%%%%%%%%%%%%%%%%%%%%%%%%%%%%%%%%%%%%%%%%%%%%%%%%%%%%%%%%%%%%%%%%%
%%%
\begin{theorem}\label{T:result}
\begin{equation}\label{E:result}
\begin{split}
& {\sf u}_{g+2}(n+2)=\\
& {\sf t}_{g}^{}(n)+{\sf d}_{g+2}(n)+4{\sf u}_{g+2}(n+1)-
3{\sf u}_{g+2}(n)+(n+1)(2n+1){\sf u}_{g+1}(n),
\end{split}
\end{equation}
where
\begin{equation*}
\begin{split}
&{d}_{g+2}(n)= \\
& 3 \sum_{g_1=0}^{g+1} \sum_{0  \leq m \leq n}
\,{\sf u}^{*}_{g_1}(m) {\sf b}_{g+1-g_1}^{}(n-m)+
\sum_{g_1} \sum_{g_2} \sum_{m_1 \geq 0} \sum_{m_2 \geq 0} {\sf u}^{*}_{g_1}(m_1)
{\sf u}^{*}_{g_2}(m_2) {\sf u}^{*}_{g+2-g_1-g_2}(n-m_1-m_2) ,
\end{split}
\end{equation*}
with
\begin{equation}
{\sf u}^{*}_{g}(n)=
\begin{cases}
0, \quad for \quad g=0 \quad and \quad n=0; \\
{\sf u}_g(n), \quad otherwise.
\end{cases}
\end{equation}
\end{theorem}

%%%
%%%%%%%%%%%%%%%%%%%%%%%%%%%%%%%%%%%%%%%%%%%%%%%%%%%%%%%%%%%%%%%%%%%%%%%%%%%%%%%%%%%
%%%

%%%
%%%%%%%%%%%%%%%%%%%%%%%%%%%%%%%%%%%%%%%%%%%%%%%%%%%%%%%%%%%%%%%%%%%%%%%%%%%%%%%%%%%
%%%
\section{Proofs}\label{S:proofs}
%%%
%%%%%%%%%%%%%%%%%%%%%%%%%%%%%%%%%%%%%%%%%%%%%%%%%%%%%%%%%%%%%%%%%%%%%%%%%%%%%%%%%%%
%%%

{\bf Proof of Lemma~\ref{L:disconnect}}
\begin{proof}
%%%
%%%%%%%%%%%%%%%%%%%%%%%%%%%%%%%%%%%%%%%%%%%%%%%%%%%%%%%%%%%%%%%%%%%%%%%%%%%%%%%%%%%%%%%%%%%%%%%
%%%%
{\bf Claim $1$:} The mapping
\begin{equation*}
\begin{split}
&\eta_{5,1}\colon U^{5,1}_{g+2, n+2}  \longrightarrow
   \dot\bigcup_{0\le g_3 \le g+1,\; 0\le j, k\le n}\left( U_{g_3, j}\times
    B_{g+1-g_3, n-j}\right),  \\
&   u_{g+2, n+2}\mapsto (u_{g_3, j}, b_{g+1-g_3, n-j}),
        \quad 0\leq g_3 \leq g+1, 1\leq j \leq n
\end{split}
\end{equation*}
is a bijection.
%%%
%%%%%%%%%%%%%%%%%%%%%%%%%%%%%%%%%%%%%%%%%%%%%%%%%%%%%%%%%%%%%%%%%%%%%%%%%%%%%%%%%%%%%%%%%%%%%%%
%%%%
We first prove that $\eta_{5,1}$ is welldefined.
For a planted unicellular map $(H, \alpha, \gamma)=u_{g+2, n+2} \in U^{5, 1}_{g+2, n+2}$ with
face
\begin{equation*}
\gamma=(R_1, \alpha(h_1^{2}), \ldots, h_1^{2},
\alpha(h_1^{3}), \ldots, h_1^{3}, h_1^{3}+1, \ldots, 2(n+1), S_1),
\end{equation*}
we employ the mapping $c_2$ of the Cutting-Lemma (Lemma~\ref{L:alpha-cutting})
in order to decompose $u_{g+2, n+2}$ into a planted $3$-cellular map, $x_{3, n+2}=(H, \alpha,
(\gamma_i)_{1\le i\le 3})$, where
\begin{equation}\label{E:bc-p1}
\begin{split}
&\gamma_1=\bar{\gamma}_{1}, \quad
\gamma_2=\bar{\gamma}_{2}, \quad\\
&\gamma_3=(R_1, h_1^{3}+1, \ldots, 2(n+1), S_1),
\end{split}
\end{equation}
where $\gamma_3$ is obtained by concatenating the sequence of half-edges of $(R_1)$,
$\bar{\gamma}_{3}$ and $(S_1)$.

For any $(H, \alpha, \gamma)=u_{g+2, n+2} \in U^{5, 1}_{g+2, n+2}$,
$H_{\bar{\gamma}_{1}}$ is closed. Since $\gamma_1= \bar{\gamma}_{1}$, the restriction
$\alpha|_{H_{\gamma_{1}}}$ is a fixed-point free involution.
Accordingly,
$(H_{\gamma_{1}}, \alpha|_{H_{\gamma_{1}}}, \gamma_1)$ is a planted unicellular map.

Since $H_{\bar{\gamma}_{2}}\cup H_{\bar{\gamma}_{3}}$ is closed and $\gamma_i$, $i=2,3$
is given in eq.~(\ref{E:bc-p1}), the restriction
$\alpha|_{H_{\gamma_1, \gamma_2}}$ is a welldefined fixed-point free involution.
Furthermore, since neither $H_{\bar{\gamma}_2}$ nor $H_{\bar{\gamma}_3}$ are closed,
$H_{\gamma_2}$ and $H_{\gamma_3}$ are not closed either.
Therefore $(H_{\gamma_1, \gamma_2}, \alpha|_{H_{\gamma_1, \gamma_2}},
(\gamma_i)_{1 \leq i \leq 2})$ is a planted bicellular map with the plants $(S_1)$ and $(h_1^3)$.

Let $u_{g_3, j}=(H_{\gamma_{1}}, \alpha|_{H_{\gamma_{1}}}, \gamma_1)$ and
$b_{g', n'}=(H_{\gamma_1, \gamma_2}, \alpha|_{H_{\gamma_1, \gamma_2}},(\gamma_i)_{1 \leq i \leq 2})$.

Suppose $u_{g+2, n+2}$, $u_{g_3, j}$ and $b_{g', n'}$ have $J$, $J_{\gamma_1}$ and $J_{b}$
vertices, respectively. Then $2-2(g+2)=J-(n+2)+1$ and $2-2g_3=J_{\gamma_1}-j+1$, whence
$$
2-2(g+1-g_3)=J-J_{\gamma_1}-(n-j)+2.
$$
Since the edges incident to plants and plants do not contribute to the number of
edges and vertices, we have
$n'=n-j, 1\leq  j  \leq n$, $J_{b}=J-J_{\gamma_1}$. As a result $b_{g', n'}$ has
genus $(g+1-g_3)$, where $0  \leq  g_3 \leq g+1$, whence $\eta_{5, 1}$ is welldefined.

We next show that $\eta_{5,1}$ is injective. In order to apply the mapping
$c_2$ of the Cutting-Lemma, we introduce
\begin{equation}
\begin{split}
&\zeta_{5, 1}((H_{\gamma_{1}}, \alpha|_{H_{\gamma_{1}}}, \gamma_1),
(H_{\gamma_1, \gamma_2} , \alpha|_{H_{\gamma_1, \gamma_2}},
(\gamma_i)_{1 \leq i \leq 2}))=
(H, \alpha, (\gamma_i)_{1\le i\le 3}),
\end{split}
\end{equation}
where $\gamma_i$ are given by eq.~(\ref{E:bc-p1}),
$\alpha=\alpha|_{H_{\gamma_1, \gamma_2, \gamma_3}}$ and
$H=H_{\gamma_{1}}\cup H_{\gamma_1, \gamma_2}$.

For any
$$
\eta_{5, 1}((H, \alpha, \gamma))=
(u_{g_3, j}, b_{g+1-g_3, n-j})\in \dot\bigcup_{0\le g_3 \le g+1,\; 0\le j, k\le n}
\left( U_{g_3, j}\times B_{g+1-g_3, n-j}\right),
$$
where $u_{g_3, j}=(H_{\gamma_{1}}, \alpha|_{H_{\gamma_{1}}}, \gamma_1)$, and
$b_{g+1-g_3, n-j}=(H_{\gamma_1, \gamma_2} , \alpha|_{H_{\gamma_1, \gamma_2}},
(\gamma_i)_{1 \leq i \leq 2})$,
we apply $\zeta_{5, 1}$. This generates the $3$-cellular map
$(H, \alpha, (\gamma_i)_{1\le i\le 3})$.
Since $u_{g_3, j}$ has $j$ edges and $b_{g+1-g_3, n-j} $ has $(n-j)$ edges, and the process
generates the edges $(\alpha(h_1^2), h_1^2)$ and $(\alpha(h_1^3), h_1^3)$, we have
$x_{3, n+2}=(H, \alpha, (\gamma_i)_{1\le i\le 3})\in X_3(n+2)$.
We can now apply $c_2$ of Lemma~\ref{L:alpha-cutting}, which induces the mapping
$c_{5, 1} \colon X_3(n+2) \longrightarrow U^{5, 1}_{g+2, n+2}$. Lemma~\ref{L:alpha-cutting}
now implies furthermore
$$
(c_{5, 1} \circ \zeta_{5, 1}) \circ \eta_{5, 1}=\text{\rm id},
$$
whence the mapping $\eta_{5,1}$ is injective.

It thus remains to prove that $\eta_{5,1}$ is surjective. This follows again from close
inspection of the proof of the Lemma~\ref{L:alpha-cutting}, which implies
$$
\eta_{5, 1} \circ (c_{5, 1} \circ \zeta_{5, 1})=\text{\rm id}.
$$
Therefore, $\eta_{5,1}$ is surjective and Claim $1$ is completed.
%%%
%%%%%%%%%%%%%%%%%%%%%%%%%%%%%%%%%%%%%%%%%%%%%%%%%%%%%%%%%%%%%%%%%%%%%%%%%%%%%%%%%%%%%%%%%%%%%%%
%%%%

Analogously we prove that $\eta_{5, 2}$ and $\eta_{5, 3}$ are injective.

%%%
%%%%%%%%%%%%%%%%%%%%%%%%%%%%%%%%%%%%%%%%%%%%%%%%%%%%%%%%%%%%%%%%%%%%%%%%%%%%%%%%%%%%%%%%%%%%%%%
%%%% Case 2
{\bf Claim $2$}: The mapping
\begin{equation*}
\begin{split}
&\eta_{5,4}: U^{5,4}_{g+2, n+2}  \longrightarrow   \bigcup_{0\le g_1, g_2\le g+2,\; 0\le j, k\le n}
U_{g_1, j}\times U_{g_2, k} \times U_{g+2-g_1-g_2, n-j-k}, \\
& u_{g+2, n+2}  \mapsto (u_{g_1, j}, u_{g_2, k}, u_{g+2-g_1-g_2, n-j-k}),
\end{split}
\end{equation*}
with $1\leq g_1, g_2 \leq g+2$ and $1\leq j, k \leq n$ is a bijection.
%%%
%%%%%%%%%%%%%%%%%%%%%%%%%%%%%%%%%%%%%%%%%%%%%%%%%%%%%%%%%%%%%%%%%%%%%%%%%%%%%%%%%%%%%%%%%%%%%%%
%%%%

We first show that $\eta_{5,4}$ is well-defined. As in the proof of Claim $1$,
we employ the Cutting-Lemma which produces a $3$-cellular map with the boundary
components $(\gamma_1,\gamma_2,\gamma_3)$.

For any $u_{g+2, n+2} \in U^{5,4}_{g+2, n+2}$, each of the $H_{\bar{\gamma}_{i}}$ is closed.
Thus the restrictions $\alpha|_{H_{\gamma_i}}$, for $i=1, 2, 3$ are
welldefined and fixed-point free involutions. As a result, $(H_{\gamma_1},
\alpha|_{H_{\gamma_1}}, \gamma_{1})$, $(H_{\gamma_1}, \alpha|_{H_{\gamma_2}}, \gamma_2)$
and $(H_{\gamma_1}, \alpha|_{H_{\gamma_3}},  \gamma_3)$ are unicellular maps, respectively.

Let $u_{g_1, j}=(H_{\gamma_1}, \alpha|_{H_{\gamma_1}}, \gamma_{1})$,
    $u_{g_2, k}=(H_{\gamma_2}, \alpha|_{H_{\gamma_2}}, \gamma_2)$
and $u_{g', n'}=(H_{\gamma_3}, \alpha|_{H_{\gamma_3}}, \gamma_3)$.
Suppose that $u_{g+2, n+2}$, $u_{g_1, j}$, $u_{g_2, k}$ and $u_{g', n'}$ have $J$,
$J_{\gamma_1}$, $J_{\gamma_2}$ and $J_{\gamma_3}$ vertices, respectively. Then
\begin{equation}\label{E:case1-g1}
2-2(g+2)=J-(n+2)+1, \quad 2-2g_1=J_{\gamma_1}-j+1, \quad \text{and}
\quad 2-2g_2=J_{\gamma_2}-k+1, \quad 1\leq  j, k \leq n.
\end{equation}
Furthermore, we have
\begin{equation*}\label{E:case2genus}
\begin{split}
&J-J_{\gamma_1}-J_{\gamma_2}-(n-j-k)+1=2-2(g+2-g_1-g_2).
\end{split}
\end{equation*}
After applying the Cutting-Lemma, $(h_1^2)$ and $(h_1^3)$ become plants, similarly
$(\alpha(h_1^2), h_1^2)$ and $(\alpha(h_1^3), h_1^3)$ become edges incident to plants.
Thus, we have $n^{'}=n-j-k, 0\le j, k\le n$ and $J_{\gamma_3}=J-2-J_{\gamma_1}-J_{\gamma_2}$
and accordingly obtain
\begin{equation*}
\begin{split}
&J-2-J_{\gamma_1}-J_{\gamma_2}-(n-j-k)+1=2-2(g+2-g_1-g_1).
\end{split}
\end{equation*}
Consequently, $u_{g^{'}, n^{'}}$ has genus $(g+2-g_1-g_2)$, where $0  \leq g_1, g_2 \leq g+2$ and
$\eta_{5, 4}$ is well-defined.

We next prove $\eta_{5, 4}$ is injective. We establish this as in Claim $1$, introducing
\begin{equation}
\begin{split}
&\zeta_{5, 4}: ((H_{\gamma_1}, \alpha|_{H_{\gamma_1}}, \gamma_{1}),
(H_{\gamma_2}, \alpha|_{H_{\gamma_2}}, \gamma_2),
(H_{\gamma_3}, \alpha|_{H_{\gamma_3}}, \gamma_3)) \mapsto
(H, \alpha, (\gamma_i)_{1\le i\le 3})\\
\end{split}
\end{equation}
where $\gamma_i$ is given in eq.~(\ref{E:bc-p1}),
$\alpha=\alpha|_{H_{\gamma_1, \gamma_2, \gamma_3}}$ and
$H=H_{\gamma_{1}}\cup H_{\gamma_1, \gamma_2}$.
Analogously, $c_2$ of Lemma~\ref{L:alpha-cutting} induces the mapping $c_{5, 4}$ and
$$
(c_{5, 4} \circ \zeta_{5, 4}) \circ \eta_{5, 4}=\text{\rm id},
$$
whence the mapping $\eta_{5,4}$ is injective.

Subjectivity of $\eta_{5, 4}$ is implied by the Cutting-Lemma which guarantees
$$
\eta_{5, 4} \circ (c_{5, 4} \circ \zeta_{5, 4})=\text{\rm id},
$$
whence Claim $2$ and the proof of the lemma is complete.
%%%
%%%%%%%%%%%%%%%%%%%%%%%%%%%%%%%%%%%%%%%%%%%%%%%%%%%%%%%%%%%%%%%%%%%%%%%%%%%%%%%%%%%%%%%%%%%%%%%
%%%%
\end{proof}
%%%
%%%%%%%%%%%%%%%%%%%%%%%%%%%%%%%%%%%%%%%%%%%%%%%%%%%%%%%%%%%%%%%%%%%%%%%%%%%%%%%%%%%%%%%%%%%%%%%%
%%%
{\bf Proof of Proposition \ref{P:tricellular}}.
\begin{proof}
We prove that the mapping
\begin{equation}
\begin{split}
& \theta\colon U^{II}_{g+2, n+2} \longrightarrow T_{g, n},\\
& u_{g+2, n+2} \mapsto t_{g, n}
\end{split}
\end{equation}
is a bijection.
As for welldefinedness, suppose $u_{g+2, n+2} \in U^{II}_{g+2, n+2}$ where
\begin{equation*}
\gamma=(R_1, \alpha(h_1^{2}), \ldots, h_1^{2},
\alpha(h_1^{3}), \ldots, h_1^{3}, h_1^{3}+1, \ldots, 2(n+1), S_1).
\end{equation*}
We use mapping $c_2$ of the Cutting-Lemma and
derive the planted $3$-cellular map, $x_{3, n+2}=(H, \alpha, (\gamma_i)_{1\le i\le 3})$,
where
\begin{equation}\label{E:bc-p3}
\begin{split}
&\gamma_1=\bar{\gamma}_{1}, \quad \gamma_2=\bar{\gamma}_{2}, \quad \\
&\gamma_3=(R_1, h_1^{3}+1, \ldots, 2(n+1), S_1).
\end{split}
\end{equation}
Here $\gamma_3$ is obtained by concatenating the sequence of half-edges contained
in $(R_1)$, $\bar{\gamma}_{3}$ and $(S_1)$.

For $u_{g+2,n+2}\in U^{II}_{g+2,n+2}$ none of the $H_{\bar{\gamma}_{i}}$, $i=1,2,3$ is
closed, whence the associated combinatorial graph of $x_{3, n+2}$ is connected.
Accordingly, $x_{3,n+2}=\theta(u_{g+2, n+2})$ is a planted tricellular map with
plants $(h_1^2)$, $(h_1^3)$ and $(S_1)$. Euler's characteristic formula implies
$t_{g, n}$ has genus $g$ and $n$ edges, whence $\theta$ is well-defined.
Injectivity and surjectivity of $\theta$ are implied by the Cutting Lemma.
\end{proof}
%%%
%%%%%%%%%%%%%%%%%%%%%%%%%%%%%%%%%%%%%%%%%%%%%%%%%%%%%%%%%%%%%%%%%%%%%%%%%%%%%%%%%%%%%%%%%%%%
%%%
%%%%%%%%%%%%%%%%%%%%%%%%%%%%%%%%%%%%%%%%%%%%%%%%%second Lemma
{\bf Proof of Lemma \ref{L:add-1edge}.}
\begin{proof}
{\bf Claim $1$:} The mapping
\begin{equation}
\begin{split}
&\eta_{1}\colon  U^1_{g+2, n+2} \longrightarrow  U_{g+2, n+1},               \\
&((H, \alpha, \gamma), (\alpha(h_1^2), h_1^2)) \mapsto
                          ((H', \alpha', \gamma'), (\alpha(h_1^2)-1, h_1^2-1)).
\end{split}
\end{equation}
is a bijection.

The contraction lemma implies that $\eta_1$ is welldefined. Injectivity of $\eta_1$ follows
by considering the mapping $\rho_1$ induced by the mapping $e_2$ of Lemma~\ref{L:contract2},
where $\rho_1((H', \alpha', \gamma'))=(H, \alpha, \gamma) \in U^1_{g+2,n+2}$.
Lemma~\ref{L:contract2} guarantees $\rho_1 \circ \eta_1=\text{\rm id}$, whence $\eta_1$ is
injective. Surjectivity of $\eta_1$ is a consequence of $\eta_1 \circ \rho_1=id$, implied by
Lemma~\ref{L:contract2}, see Fig.~\ref{F:one}.
%%%
%%%%%%%%%%%%%%%%%%%%%%%%%%%%%%%%%%%%%%%%%%%%%%%%%%%%%%%%%%%%%%%%%%%%%%%%%%%%%%%%%%%%%%%%%%%%%%%%%%
%%%
\begin{figure}[ht]
\begin{center}
\includegraphics[width=0.4\columnwidth]{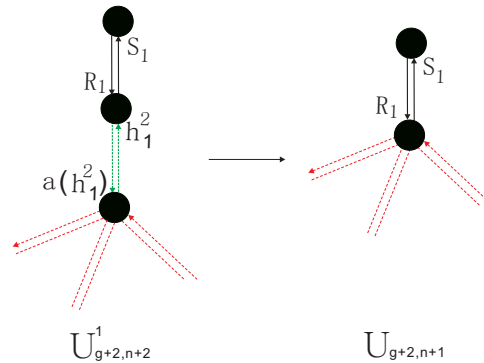}
\end{center}
\caption{\small The mapping $\eta_1$ and the one-sided edge $(\alpha(h_1^2), h_1^2))$ (green).}\label{F:one}
\end{figure}
%%%
%%%%%%%%%%%%%%%%%%%%%%%%%%%%%%%%%%%%%%%%%%%%%%%%%%%%%%%%%%%%%%%%%%%%%%%%%%%%%%%%%%%%%%%%%%%%%%%%%%
%%%

The proof that $\eta_{j}\colon  U^j_{g+2, n+2} \longrightarrow  U_{g+2, n+1}$ is a bijection for
$j=2,3,4$ follows analogously, see Fig.~\ref{F:2-3-4}.

%%%
%%%%%%%%%%%%%%%%%%%%%%%%%%%%%%%%%%%%%%%%%%%%%%%%%%%%%%%%%%%%%%%%%%%%%%%%%%%%%%%%%%%%%%%%%%%%%%%%%%
%%%
\begin{figure}[ht]
\begin{center}
\includegraphics[width=0.85\columnwidth]{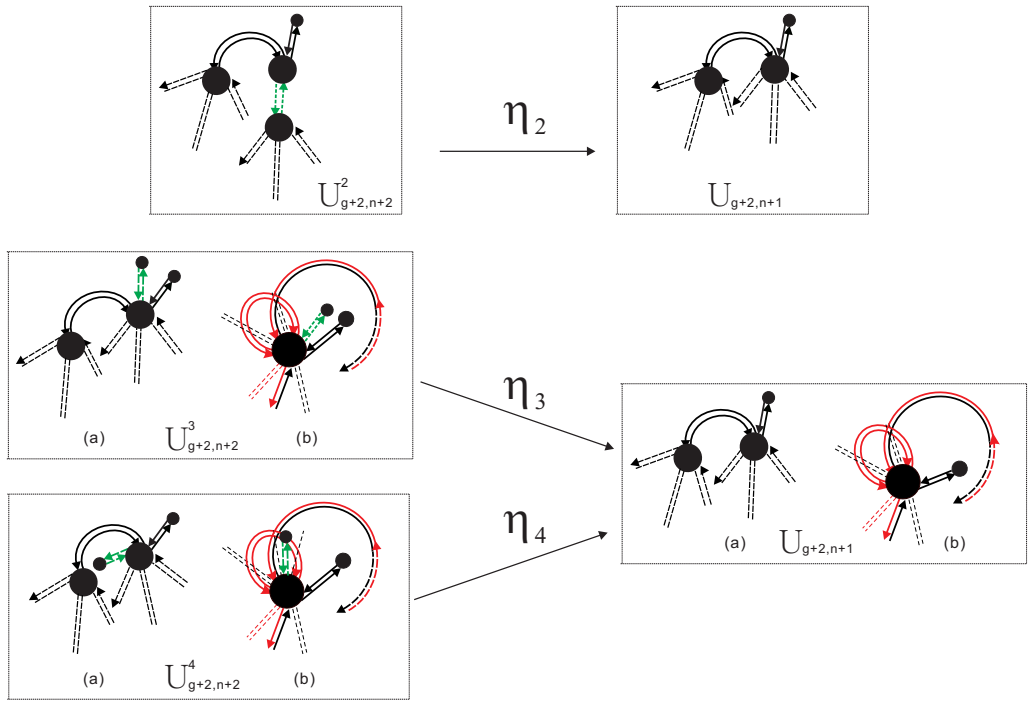}
\end{center}
\caption{\small The mappings $\eta_2$, $\eta_3$ and $\eta_4$.}\label{F:2-3-4}
\end{figure}
%%%
%%%%%%%%%%%%%%%%%%%%%%%%%%%%%%%%%%%%%%%%%%%%%%%%%%%%%%%%%%%%%%%%%%%%%%%%%%%%%%%%%%%%%%%%%%%%%%%%%%
%%%
\end{proof}
%%%
%%%%%%%%%%%%%%%%%%%%%%%%%%%%%%%%%%%%%%%%%%%%%%%%%%%%%%%%%%%%%%%%%%%%%%%%%%%%%%%%%%%%%%%%%%%%%%%%%%
%%%
{\bf The proof of Lemma \ref{L:add-2edge}.}
\begin{proof}
{\bf Claim $1$:} The mapping
\begin{equation}
\begin{split}
&\eta_{5}\colon  U^{m, 1}_{g+2, n+2} \longrightarrow  U_{g+2, n}, \\
& ((H, \alpha, \gamma), (\alpha(h_1^2), h_1^2), (\alpha(h_1^3), h_1^3)) \mapsto
((H', \alpha', \gamma'), (\alpha(h_1^2)-1, h_1^2-1), (\alpha(h_1^3)-1, h_1^3-1))
\end{split}
\end{equation}
is a bijection.

We first prove that $\eta_5$ is welldefined.
Consider $(H, \alpha, \gamma) \in U^{m, 1}_{g+2, n+2}$ together with two one-sided edges,
$(\alpha(h_1^2), h_1^2)$, $(\alpha(h_1^3), h_1^3)$, such that $h_1^{1}$ and $h_1^{2}$ are
incident to $v_1$, $\alpha(h_1^{1})$ and $\alpha(h_1^{2})$ are incident to $v_i$ and $v_j$.

We then apply Lemma~\ref{L:contract2} to a $(H, \alpha, \gamma) \in U^{m,1}_{g+2, n+2}$
together with the one-side edge $(\alpha(h_1^2), h_1^2)$. We iterate applying
Lemma~\ref{L:contract2} w.r.t.~the edge $(\alpha(h_1^3), h_1^3)$.
By definition of Lemma~\ref{L:contract2} this generates the unicellular map
$(H', \alpha', \gamma')$ of genus $(g+2)$ having $n$ edges
with distinguished four half-edges $\alpha(h_1^2)-1$, $h_1^2-1$,
$\alpha(h_1^3)-1$ and $h_1^3-1$.

Since Lemma~\ref{L:contract2} preserves genus, $\eta_5((H, \alpha, \gamma))$ has
genus $(g+2)$ and $n$ edges, whence $\eta_5$ is well-defined.

We next prove $\eta_5$ is injective. Suppose we have a unicellular map
$(H', \alpha', \gamma') \in U_{g+2, n}$ with four distinguished half-edges
$\alpha(h_1^2)-1$, $h_1^2-1$, $\alpha(h_1^3)-1$ and $h_1^3-1$. We observe that
the mapping $m_2$ constructed in Lemma~\ref{L:contract2} allows us to obtain
a mapping $\rho_5\colon U_{g+2,n} \to U^{m, 1}_{g+2, n+2}$ such that
$$
\rho_5\circ \eta_5=\text{\rm id},
$$
whence injectivity.

Surjectivity follows by computing $\eta_5 \circ \rho_5=\text{\rm id}$.

The proof that
$$
\eta_{6}\colon  U^{m, 2}_{g+2, n+2} \longrightarrow  U_{g+2, n}
\quad \text{\rm and}\quad
\eta_{7}\colon  U^{m, 3}_{g+2, n+2} \longrightarrow  U_{g+2, n}
$$
are bijections is analogous, see Fig.~\ref{F:5-6-7}.

%%%
%%%%%%%%%%%%%%%%%%%%%%%%%%%%%%%%%%%%%%%%%%%%%%%%%%%%%%%%%%%%%%%%%%%%%%%%%%%%%%%%%%%%%%%%%%%%%%%%%%
%%%
\begin{figure}[ht]
\begin{center}
\includegraphics[width=0.55\columnwidth]{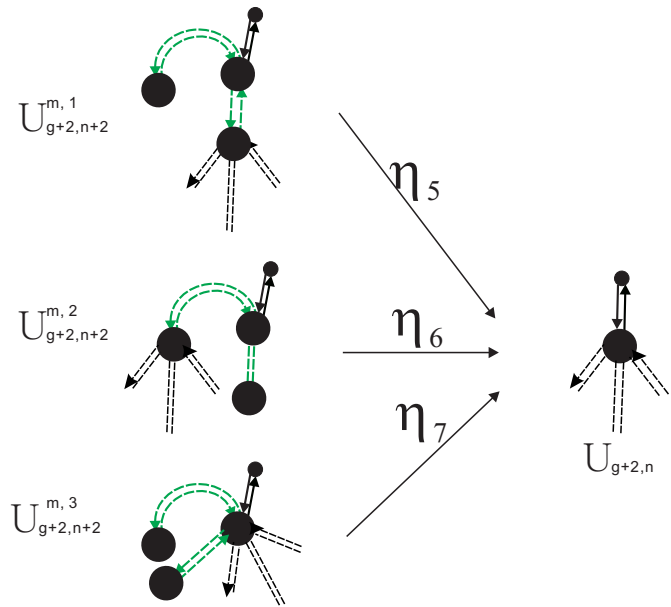}
\end{center}
\caption{\small The mappings $\eta_5$, $\eta_6$ and $\eta_7$.}\label{F:5-6-7}
\end{figure}
%%%
%%%%%%%%%%%%%%%%%%%%%%%%%%%%%%%%%%%%%%%%%%%%%%%%%%%%%%%%%%%%%%%%%%%%%%%%%%%%%%%%%%%%%%%%%%%%%%%%%%
%%%
\end{proof}

%%%
%%%%%%%%%%%%%%%%%%%%%%%%%%%%%%%%%%%%%%%%%%%%%%%%%%%%%%%%%%%%%%%%%%%%%%%%%%%%%%
%%%
{\bf The proof of Proposition \ref{P:insert2edges}}.

\begin{proof}
Proposition~\ref{P:insert2edges} follows directly from Lemma~\ref{L:beta-cutting}.
\end{proof}
%%%
%%%%%%%%%%%%%%%%%%%%%%%%%%%%%%%%%%%%%%%%%%%%%%%%%%%%%%%%%%%%%%%%%%%%%%%%%%%%%%
%%%

{\bf The proof of Theorem~\ref{T:result}.}
\begin{proof}
According to Lemma~\ref{L:disconnect}, Lemma~\ref{L:add-1edge}, Lemma~\ref{L:add-2edge},
Proposition~\ref{P:tricellular} and Proposition~\ref{P:insert2edges},
we have
\begin{itemize}
\item ${\sf u}_{g+2}^{0, i}(n+2)= {\sf u}_{g+2}(n+1)$,
for $1 \leq i \leq 4$,
\item ${\sf u}^i_{g+2}(n+2)={\sf u}_{g+2}(n+1)$,
for $1 \leq i \leq 4$,
\item ${\sf u}^{m, j}_{g+2}(n+2)={\sf u}_{g+2}(n)$,
for $1 \leq j \leq 3$,
\item ${\sf u}^{II}_{g+2}(n+2)={\sf t}_g^{}(n)$,
\item ${\sf u}^{B}_{g+2}(n+2)=(n+1)(2n+1) {\sf u}_{g+1}(n)$.
\end{itemize}
According to eq.~(\ref{E:AB}), eq.~(\ref{E:I-II}) and eq.~(\ref{F:disU}) we have
\begin{equation}
\begin{split}
&{\sf u}_{g+2, n+2}\\
=&{\sf u}^{I}_{g+2}(n+2)+{\sf u}^{II}_{g+2}(n+2)+{\sf u}^B_{g+2}(n+2)\\
%=&(1+{\sf u}^2_{g+2}(n+2)+{\sf u}^1_{g+2}(n+2)+{\sf u}^{>2}_{g+2}(n+2))
%+ {\sf t}_g^{}(n) + (n+1)(2n+1) {\sf u}_{g+1}(n)\\
=& 1+ 2{\sf u}_{g+2}(n+1)+{\sf u}^{>2}_{g+2}(n+2))
+ {\sf t}_g^{}(n) + (n+1)(2n+1) {\sf u}_{g+1}(n).
\end{split}
\end{equation}
Furthermore, according to eq.~(\ref{E:geq2}), eq.~(\ref{E:U-g2-3}) and
eq.~(\ref{E:U-g3-4}), we have
\begin{equation}
\begin{split}
& {\sf u}^{>2}_{g+2}(n+2)\\
=&{\sf u}^5_{g+2}(n+2)+{\sf u}^{>2, 3}_{g+2}(n+2)
+{\sf u}_{g+2}^{>2, 4}(n+2)\\
%=& {\sf d}_{g+2}(n)+({\sf u}^3_{g+2}(n+2)-{\sf u}_{g+2}^{m, 1}(n+2))
%+({\sf u}^4_{g+2}(n+2)-{\sf u}_{g+2}^{m, 2}(n+2))-{\sf u}_{g+2}^{m, 3}(n+2)\\
=& {\sf d}_{g+2}(n)+2{\sf u}_{g+2}(n+1)-3{\sf u}_{g+2}(n),
\end{split}
\end{equation}
which establishes eq.~(\ref{E:result}).
\end{proof}

%%%%%%%%%%%%%%%%%%%%%%%%%%%%%%%%%%%%%%%%%%%%%%%%%%%%%%
%%%%%%%%%%%%%%%%%%%%%%%%%%%%%%%%%%%%%%%%%%%%%%%%%%%%%%

\section{ Acknowledgments.}
Many thanks to our group at SDU for discussions and suggestions.
We furthermore acknowledge the financial support of the Future and Emerging
Technologies (FET) programme within the Seventh Framework Programme (FP7) for
Research of the European Commission, under the FET-Proactive grant agreement
TOPDRIM, number FP7-ICT-318121.

\bibliographystyle{plain}      % basic style, author-year citations

\end{document}